\documentclass{article}
\usepackage[margin=1.2in]{geometry} 
\usepackage{amsmath,amsthm,amssymb,amsfonts}
\usepackage{xcolor}
\usepackage{mathrsfs}
\usepackage{palatino}
\usepackage{verbatim}
\usepackage{float}
\usepackage{graphicx}
\usepackage{tikz-cd}
\usepackage{marvosym}
\usepackage{hyperref}
\usepackage{ytableau}
\usepackage{multicol}
\usepackage{dynkin-diagrams}
\usepackage[shortlabels]{enumitem}

\newcommand{\N}{\mathbb{N}}
\newcommand{\Z}{\mathbb{Z}}
\newcommand{\R}{\mathbb{R}}

\newcommand{\C}{\mathbb{C}}
\newcommand{\CP}{\mathbb{CP}}
\newcommand{\sphere}{\mathbb S}

\DeclareMathOperator{\Id}{Id}

\DeclareMathOperator{\OO}{O}

\DeclareMathOperator{\PGL}{PGL}

\DeclareMathOperator{\SO}{SO}
\DeclareMathOperator{\HH}{H}

\DeclareMathOperator{\RRef}{Ref}
\DeclareMathOperator{\Bl}{Bl}

\DeclareMathOperator{\Aut}{Aut}
\DeclareMathOperator{\Mod}{Mod}
\DeclareMathOperator{\Diff}{Diff}
\DeclareMathOperator{\Homeo}{Homeo}
\DeclareMathOperator{\Isom}{Isom}

\DeclareMathOperator{\Stab}{Stab}
\DeclareMathOperator{\dJ}{\mathcal J}
\DeclareMathOperator{\supp}{supp}
\DeclareMathOperator{\proj}{pr}
\DeclareMathOperator{\diag}{diag}

\newtheorem{thm}{Theorem}[section]
\newtheorem{cor}[thm]{Corollary}
\newtheorem{prop}[thm]{Proposition}
\newtheorem{lem}[thm]{Lemma}

\newtheorem{defn}[thm]{Definition}

\theoremstyle{remark}
\newtheorem{rmk}[thm]{Remark}

\begin{document}
\title{Mapping classes fixing an isotropic homology class of minimal genus $0$ in rational $4$-manifolds}
\author{Seraphina Eun Bi Lee}
\date{}
\maketitle
\begin{abstract}
For any $N \geq 1$, let $M_N$ denote the rational $4$-manifold $\CP^2 \# N \overline{\CP^2}$. In this paper we study the stabilizer $\Stab(w)$ of a primitive, isotropic class $w\in H_2(M_N; \Z)$ of minimal genus $0$ under the natural action of the topological mapping class group $\Mod(M_N)$ on $H_2(M_N; \Z)$. Although most elements of $\Stab(w)$ cannot be represented by homeomorphisms that preserve any Lefschetz fibration $M_N \to \Sigma$, we show that any element of $\Stab(w)$ can be represented by a diffeomorphism that \emph{almost preserves} a holomorphic, genus-$0$ Lefschetz fibration $\proj: M_N \to \CP^1$ whose generic fibers represent the homology class $w$. We also answer the Nielsen realization problem for a certain maximal torsion-free, abelian subgroup $\Lambda_w$ of $\Mod(M_N)$ by finding a lift of $\Lambda_w$ to $\Diff^+(M_N) \leq \Homeo^+(M_N)$ under the quotient map $q: \Homeo^+(M_N)\to \Mod(M_N)$ which can be made to almost preserve $\proj: M_N \to \CP^1$. All results of this paper also hold for every primitive, isotropic class $w \in H_2(M_N; \Z)$ if $N \leq 8$ because any such class has minimal genus $0$. 
\end{abstract}

\section{Introduction}\label{sec:introduction}

The (topological) \emph{mapping class group} $\Mod(M)$ of a closed, oriented manifold $M$ is the group
\[
	\Mod(M) := \pi_0(\Homeo^+(M))
\]
of isotopy classes of homeomorphisms of $M$. There is a natural action of $\Mod(M)$ on $H_2(M; \Z)$ preserving the intersection form $Q_M$ and we consider the stabilizer $\Stab(w) \leq \Mod(M)$ of any class $w \in H_2(M; \Z)$.

Suppose $M$ is a smooth, simply-connected $4$-manifold. If $w \in H_2(M; \Z)$ is a nonzero homology class with self-intersection $0$ then $w$ is called \emph{isotropic}. One way in which isotropic classes arise are as the homology class of the generic fibers of a Lefschetz fibration $p: M \to \Sigma$ where $\Sigma$ is a closed, oriented surface. 

In some settings, elements $g \in \Stab(w) \leq \Mod(M)$ are known to admit representative maps $\varphi$ that preserve some Lefschetz fibration $p: M \to \Sigma$ whose generic fibers represent the homology class $w$, meaning that there exists some diffeomorphism $\psi$ of $\Sigma$ such that $p \circ \varphi = \psi \circ p$ and $[\varphi] = g$. For example, Gizatullin (\cite{gizatullin}) showed that any parabolic automorphism of a compact K\"ahler surface $M$ must preserve some elliptic fibration $M \to \Sigma$ (also see \cite[Proposition 1.4]{cantat} or \cite[Theorem 4.3, Appendix]{diller--favre}). Smoothly, the cases of rational elliptic surfaces and K3 surfaces are studied in forthcoming work of Farb--Looijenga \cite{farb--looijenga-elliptic}; they show, for example, that any $g \in \Stab(w)$ can be represented by a diffeomorphism preserving the fibers of some holomorphic elliptic fibration $M \to \CP^1$.

In this paper we study representative maps of the stabilizers of isotropic classes of \emph{rational manifolds} $M$ and their relationships to genus-$0$ Lefschetz fibrations $M \to \Sigma$. More specifically, we study manifolds of the form
\[
	M_N := \CP^2 \#N \overline{\CP^2} \qquad \text{ for }N \geq 1,
\]
which are the underlying smooth $4$-manifolds of the blowup of $\CP^2$ at $N$ points. If $N \leq 8$, all primitive, isotropic classes $w \in H_2(M_N; \Z)$ are represented by generic fibers of a genus-$0$, holomorphic Lefschetz fibration $p: M_N \to \CP^1$. We sometimes refer to such a Lefschetz fibration as a \emph{conic bundle structure on} $M_N$. Note that these Lefschetz fibrations are not relatively minimal unless $N = 1$. See Section \ref{sec:dejonq}.

\bigskip
\noindent
{\bf{Representing $\Stab(w)$ by diffeomorphisms.}} Let $N \geq 1$ and let $w \in H_2(M_N; \Z)$ be any primitive, isotropic class of minimal genus $0$. Although any such class $w$ is represented by a generic fiber of a genus-$0$ Lefschetz fibration $p: M_N \to \CP^1$, it is not hard to show that no $\varphi \in \Homeo^+(M_N)$ with $[\varphi] \in \Stab(w)$ can preserve such a fibration $p$ if $[\varphi]$ has infinite order in $\Mod(M_N)$.
\begin{prop}\label{prop:no-lift-preserving-any-pi}
Let $N \geq 1$ and let $w \in H_2(M_N; \Z)$ be a primitive, isotropic class of minimal genus $0$. Let $\varphi \in \Homeo^+(M_N)$ represent an infinite-order mapping class $[\varphi] \in \Stab(w) \leq \Mod(M_N)$. There does not exist any Lefschetz fibration $p: M_N \to \Sigma$ where $\Sigma$ is a closed, oriented surface such that $\varphi$ preserves $p$, i.e. such that there exists a homeomorphism $h: \Sigma \to \Sigma$ such that $p \circ \varphi = h \circ p$. 
\end{prop}
For a proof, see Section \ref{sec:dejonq}. In this paper we ask instead that any diffeomorphism representing any infinite-order mapping class $f \in \Stab(w) \leq \Mod(M_N)$ \emph{almost preserves} some Lefschetz fibration $p: M_N \to \Sigma$.
\begin{defn}[{\bf{Almost preserving a Lefschetz fibration}}]\label{defn:almost-preserving}
A group of diffeomorphisms $G \leq \Diff^+(M)$ \emph{almost preserves} a Lefschetz fibration $p: M \to \Sigma$ if the elements of $G$ act on the fibers of $p$ outside of disjoint neighborhoods the singular fibers of $p$. More precisely, there exist
\begin{enumerate}[(a)]
\item disjoint, open neighborhoods $V_1, \dots, V_{m} \subseteq \Sigma$ of the images of the singular points $z_1, \dots, z_{m} \in \Sigma$, and
\item a homomorphism $i: G \to \Diff^+\left(\Sigma - \bigcup_{k=1}^{m} V_k\right)$
\end{enumerate}
such that for all $\varphi \in G$, the following commutes: 
\begin{figure}[H]
\centering
\begin{tikzcd}
M-\bigcup_{k=1}^{m} p^{-1}(V_k) \arrow[r, "\varphi"] \arrow[d, "p"] & M-\bigcup_{k=1}^{m} p^{-1}(V_k) \arrow[d, "p"] \\
\Sigma - \bigcup_{k=1}^{m} V_k \arrow[r, "i(\varphi)"]                     & \Sigma - \bigcup_{k=1}^{m} V_k                       
\end{tikzcd}
\end{figure}
\end{defn}

With Definition \ref{defn:almost-preserving} in hand, our first theorem shows a contrast with Proposition \ref{prop:no-lift-preserving-any-pi}:
\begin{thm}[{\bf{Mapping classes fixing an isotropic class}}]\label{thm:individual-stab-v}
Let $N \geq 1$ and let $w \in H_2(M_N; \Z)$ be a primitive, isotropic class of minimal genus $0$. For any $h\in \Stab(w)$, there exists $\varphi\in\Diff^+(M_N)$ almost preserving a holomorphic genus-$0$ Lefschetz fibration $p: M_N \to \CP^1$ whose generic fiber represents the homology class $w$ such that $[\varphi] = h \in \Mod(M_N)$.
\end{thm}
If $N \leq 8$, Theorem \ref{thm:individual-stab-v} holds for any primitive, isotropic class $w \in H_2(M_N; \Z)$ because any such class has minimal genus $0$. See Corollary \ref{cor:individual-stab-v}.

We can also consider subgroups of $\Stab(w)$ rather than individual elements. The next theorem concerns a certain finite-index, abelian, torsion-free subgroup of $\Stab(w)$ which we now define. For any $N \geq 2$, any element of $\Stab(w)$ must preserve the following $\Z$-submodule of $H_2(M_N; \Z)$:
\[
	w^\perp := \{ w_0 \in H_2(M_N; \Z) : Q_{M_N}(w,w_0) = 0 \} \cong \Z^N.
\]
Therefore, $\Stab(w)$ acts on the lattice $(w^\perp/\Z\{w\}, \overline{Q_{M_N}})$ where $\overline{Q_{M_N}}$ is the unimodular, symmetric, bilinear form on $w^\perp/\Z\{w\}$ induced by $Q_{M_N}$. Because $(H_2(M_N; \Z), Q_{M_N})$ has signature $(1, N)$, $(w^\perp/\Z\{w\}, \overline{Q_{M_N}})$ must be negative definite of rank $N-1$.
\begin{defn}\label{defn:lambda}
Let $\Lambda_w$ denote the kernel of the map $\Stab(w) \to \Aut(w^\perp/\Z\{w\}, \overline{Q_{M_N}})$. 
\end{defn}
There is an identification of $\Lambda_w$ with the subgroup of even elements of the lattice $(w^\perp/\Z\{w\}, \overline{Q_{M_N}})$, and $\Stab(w)$ fits into a split short exact sequence
\[
0 \to \underbrace{\Lambda_w}_{\cong \Z^{N-1} \leq w^\perp/\Z\{w\}} \to \Stab(w) \to\Aut(w^\perp/\Z\{w\}, \overline{Q_{M_N}}) \to 0.
\]
Two properties of $\Lambda_w$ are that it is a maximal torsion-free, abelian subgroup of $\Mod(M_N)$ and that it has finite index in $\Stab(w)$. See Lemmas \ref{lem:eichler} and \ref{lem:ses}. 
\begin{thm}[{\bf{Realizing $\Lambda_w$ by diffeomorphisms}}]\label{thm:parabolics}
Let $N \geq 2$ and let $w \in H_2(M_N; \Z)$ be a primitive, isotropic class of minimal genus $0$. There exists a homomorphism $\rho_w: \Lambda_w \to \Diff^+(M_{N})$ such that the following diagram commutes: 
\begin{figure}[H]
\centering
\begin{tikzcd}
& \Diff^+(M_{N}) \arrow[d, "q"] \\
\Lambda_w \arrow[ru, "\rho_w", dashed] \arrow[r, hook] & \Mod(M_{N})                  
\end{tikzcd}
\end{figure}
Moreover, the image $\rho_w(\Lambda_w)$ almost preserves a holomorphic genus-$0$ Lefschetz fibration $p: M_{N} \to \CP^1$ whose generic fiber represents the homology class $w$.
\end{thm}
Similarly as with Theorem \ref{thm:individual-stab-v}, Theorem \ref{thm:parabolics} holds for any primitive, isotropic class $w \in H_2(M_N; \Z)$ if $2 \leq N \leq 8$. See Corollary \ref{cor:parabolics}.

One way to interpret the results of this paper is via the natural action of (an index-$2$ subgroup of) $\Mod(M_N)$ on $\mathbb H^N$ and the classification of hyperbolic isometries into three types: elliptic, parabolic, and hyperbolic. Infinite-order elements of the stabilizer $\Stab(w)$ for an isotropic class $w \in H_2(M_N; \Z)$ are precisely the elements of $\Mod(M_N)$ acting by parabolic isometries on $\mathbb H^N$ (Lemma \ref{lem:infinite-order-parabolics}). Therefore the following is an immediate corollary of Theorem \ref{thm:individual-stab-v}.
\begin{cor}\label{cor:individual-parabolics}
Let $2 \leq N \leq 8$. If $g \in \Mod(M_N)$ acts by a parabolic isometry on $\mathbb H^N$ then there exists $\varphi \in \Diff^+(M_N)$ with $[\varphi] = g$ and almost preserving a holomorphic genus-$0$ Lefschetz fibration $p: M_N \to \CP^1$.
\end{cor}

\bigskip
\noindent
{\bf{Related work.}} The relationship between mapping classes of $4$-manifolds fixing an isotropic class and Lefschetz fibrations with the prescribed generic fiber has been studied in some settings. As mentioned above, see Gizatullin \cite{gizatullin} for the case of compact, K\"ahler surfaces and elliptic fibrations and Farb--Looijenga \cite{farb--looijenga-elliptic} for the case of rational elliptic and K3 manifolds; \cite{farb--looijenga-elliptic} was an inspiration for this current paper.

Automorphisms preserving a conic bundle structure also play an important role in the study of finite groups of automorphisms of $M_N$. An example of such a complex automorphism is the \emph{de Jonqui\'eres} involution, which is a main tool for this paper. Some examples of work in this direction include the classification of order-$2$ birational automorphisms of $\CP^2$ up to conjugacy (Bertini \cite{bertini}, Bayle--Beauville \cite{bayle--beauville}) and finite subgroups of birational automorphisms of $\CP^2$ in general (Dolgachev--Iskovskikh \cite{dolgachev--iskovskikh}) in the complex category and a study of finite groups of symplectomorphisms of rational surfaces (Chen--Li--Wu \cite{chen--li--wu}) in the symplectic category. 

\bigskip
\noindent
{\bf{Organization of the paper.}} In Section \ref{sec:isotropic}, we recall relevant facts about the mapping class group $\Mod(M_N)$ of rational manifolds and deduce basic facts about isotropic classes $w \in H_2(M_N; \Z)$, including the proof of Proposition \ref{prop:no-lift-preserving-any-pi}. In Section \ref{sec:parabolics}, we prove Theorem \ref{thm:parabolics} by explicitly constructing the necessary diffeomorphisms. Using these diffeomorphisms from Section \ref{sec:parabolics}, we prove Theorem \ref{thm:individual-stab-v} in Section \ref{sec:individual-stab-v}. 

\bigskip
\noindent
{\bf{Acknowledgments.}} I would like to thank Benson Farb for suggesting this problem, for his continuous encouragement, advice, and guidance throughout this project, and for many helpful comments on an earlier draft of this paper. I would also like to thank Benson and Eduard Looijenga for sharing their results on rational elliptic and K3 manifolds in their forthcoming work \cite{farb--looijenga-elliptic} with me. I would also like to thank Carlos A. Serv\'an for many useful conversations about rational $4$-manifolds and Lefschetz fibrations. 

\section{Isotropic homology classes and their stabilizers in $\Mod(M_N)$}\label{sec:isotropic}

In this section we collect useful properties of the mapping class groups of $4$-manifolds, isotropic classes in $H_2(M_N; \Z)$, and certain Lefschetz fibrations.

\subsection{Mapping class group of $M_N$}
For any $4$-manifold $M$, let $Q_M$ denote the intersection form on $H_2(M; \Z)$. The form $Q_M$ is an integral, unimodular, nondegenerate, symmetric bilinear form, and the lattice $(H_2(M; \Z), Q_M)$ is denoted by $\HH_M$. The automorphism group of the lattice $\HH_M$ is denoted $\OO(\HH_M)$.

The mapping class group $\Mod(M) := \pi_0(\Homeo^+(M))$ of a closed, oriented, simply connected $4$-manifold $M$ is computable due to the following theorems of Freedman and Quinn.
\begin{thm}[Freedman \cite{freedman}, Quinn \cite{quinn}]\label{thm:freedman--quinn}
Let $M^4$ be a closed, oriented, and simply connected manifold. The map
\[
	\Phi: \Mod(M) \to \OO(\HH_M)
\]
given by $\Phi: [\varphi] \mapsto \varphi_*$ is an isomorphism of groups.
\end{thm}

The Mayer--Vietoris sequence implies that $H_2(M_N; \Z) = H_2(\CP^2; \Z) \oplus H_2(\overline{\CP^2}; \Z)^{\oplus N}$ and gives the usual $\Z$-basis $\{H, E_1, \dots, E_N\}$. The intersection form $Q_{M_N}$ is given by the diagonal, $(N+1)\times(N+1)$ matrix
\[
	\diag(1, -1, \dots, -1)
\]
with respect to the $\Z$-basis $\{H, E_1, \dots, E_N\}$. On the other hand, there is a natural $\Z$-basis 
\begin{equation}\label{eqn:defn-homology}
	\{ s, v, e_1, \dots, e_{N-1}\}
\end{equation}
of $H_2((\CP^1\times\CP^1) \# (N-1)\overline{\CP^2}; \Z)$ via the Mayer--Vietoris sequence; here, $s$ and $v$ correspond to the first and second factors of $\CP^1 \times \CP^1$ respectively. There is a diffeomorphism $(\CP^1\times\CP^1) \# (N-1)\overline{\CP^2} \cong M_N$ for all $N \geq 2$ giving an identification
\[
	v = H-E_1, \quad s = H-E_2, \quad e_1 = H-E_1-E_2, \quad e_k = E_{k+1} \text{ for all }2 \leq k \leq N-1.
\]
In this paper, we will mostly work with the $\Z$-basis $\{s, v, e_1, \dots, e_{N-1}\}$ of $H_2(M_N; \Z)$.

Therefore by Freedman--Quinn (Theorem \ref{thm:freedman--quinn}),
\[
	\Mod(M_{N}) \cong\OO(1,N)(\Z) := \OO(\HH_{M_N})
\]
We will identify $\OO(\HH_{M_N})$ and $\Mod(M_N)$ throughout this paper. 

On the other hand, consider $\mathbb E^{1,N} := (\R^{N+1}, Q_N)$ where $Q_N$ is the diagonal bilinear symmetric form of signature $(1,N)$:
\[
	Q_N((x_0, x_1, \dots, x_N), (y_0, y_1, \dots, y_N)) = x_0 y_0 - x_1 y_1 - \dots - x_N y_N.
\] 
There is a natural identification of $\R^{N+1}$ with the $\R$-span of the $\Z$-basis $\{H, E_1, \dots, E_N\}$ of $H_2(M_N; \Z)$ which makes the $\R$-bilinear extension of $Q_{M_N}$ coincide with $Q_N$. The hyperboloid model for $\mathbb H^N$ sits in $\mathbb E^{1,N}$ by
\[
	\mathbb H^N = \{w = (w_0, w_1, \dots, w_N)\in \R^{N+1} : Q_N(w,w) = 1, \, w_0 > 0\}.
\]
where the Riemannian metric is defined by the restriction of $Q_N$ to $\mathbb H^N$ (see \cite[Chapter 2]{thurston}). Because $\OO(1,N)(\Z)$ acts on $\R^{N+1}$ and preserves $Q_N$, it contains an index-$2$ subgroup $\OO^+(1,N)(\Z)$ acting by isometries on $\mathbb H^N$.

The boundary sphere of $\mathbb H^N$ corresponds to 
\[
	\partial \mathbb H^N = \{w = (w_0, w_1, \dots, w_N) \in \R^{N+1} : Q_N(w,w) = 0, \, w_0 > 0\} /\sim
\]
where $aw \sim w$ for all $a \in \R_{>0}$. Parabolic isometries of $\mathbb H^N$ are those that fix a unique point of $\partial \mathbb H^n$. By \cite[Problem 2.5.24(g)]{thurston}, parabolic isometries not only preserve some line in $\R^{N+1}$ but fix it pointwise. Moreover, parabolic isometries in $\OO^+(1,N)(\Z)$ must fix a nonzero, isotropic vector with integral entries, i.e. some nonzero $w \in H_2(M_{N}; \Z)$ with $Q_{M_N}(w,w) = 0$. The following lemma shows that the converse is true as well. 
\begin{lem}\label{lem:infinite-order-parabolics}
Let $N \geq 2$. An element $f \in \OO(\HH_{M_N})\cap \Isom(\mathbb H^N)$ acts by a parabolic isometry if and only if $f$ has infinite order and there exists some primitive, isotropic class $w \in H_2(M_N; \Z)$ such that $f \in \Stab(w) \in \OO(\HH_{M_N})$.
\end{lem}
\begin{proof}
One direction holds by the discussion preceding the statement of the lemma, so it suffices to prove that if $f \in \Stab(w)$ has infinite order then $f$ acts on $\mathbb H^N$ by a parabolic isometry. 

Let $w_0 \in \mathbb E^{1,N}$ be an isotropic vector such that $f(w_0) = \lambda w_0$ for some $\lambda \in \R$. If $w_0 \in \R\{w\}^\perp$ then $w_0$ must be a scalar multiple of $w$ because the restriction of $Q_{M_N}$ to $\R\{w\}^\perp / \R\{w\}$ is negative definite. If $Q_{M_N}(w, w_0) =: a \neq 0$ then $\lambda = 1$ because
\[
	a = Q_{M_{N}}(w, w_0) = Q_{M_N}(w, f(w_0)) = \lambda a.
\]
Then $f(aw + w_0) = aw+w_0$ and 
\[
	Q_{M_N}(aw+w_0, aw+w_0) = 2a Q_{M_N}(w,w_0) = 2a^2 > 0. 
\]
A scalar multiple of $aw+w_0$ lies in $\mathbb H^{N}$, meaning $f$ acts on $\mathbb H^{N}$ by an elliptic isometry, and all such isometries of $\mathbb H^{N}$ in $\OO(\HH_{M_N})$ have finite order. Therefore, $w_0$ must be a scalar multiple of $w$ and hence $f$ fixes a unique point in $\partial \mathbb H^N$.
\end{proof}

\subsection{Primitive, isotropic classes $w\in H_2(M_N; \Z)$ and $\Stab(w) \leq \Mod(M_N)$}

Consider lattices $(L, Q)$, where $L \cong \Z^r$ as an abelian group for some $r \in \N$ and $Q$ is an integral, unimodular, nondegenerate, symmetric, bilinear form on $L$. For each primitive isotropic vector $w \in L$, there exists $u \in L$ such that $Q(w,u) = 1$ by unimodularity of $Q$. There is an orthogonal decomposition 
\[
	L = \Z\{u, w\} \oplus \Z\{u, w\}^\perp
\]
to which $Q$ restricts to a unimodular form on each factor. The restriction of $Q$ to $\Z\{u, w\}$ has signature $(1,1)$. Note that $\Z\{u, w\}^\perp$ is a lift of $w^\perp/\Z\{w\}$ under the natural quotient $w^\perp \to w^\perp/\Z\{w\}$. This means that $(\Z\{u, w\}^\perp, Q|_{\Z\{u, w\}^\perp})$ is isomorphic as a lattice to $(w^\perp/\Z\{w\}, \overline Q)$ via this quotient, where $\overline Q$ is the induced bilinear form on $w^\perp/\Z\{w\}$. We fix the above notation throughout this section.
\begin{lem}\label{lem:stab-v-determined}
Let $w \in L$ be a primitive isotropic vector. If $h_1, h_2 \in \Stab(w) \leq \OO(L, Q)$ and $h_1|_{w^\perp} = h_2|_{w^\perp}$ then $h_1 = h_2$. In particular, for any $h_1, h_2 \in \Stab(v) \leq \OO(\HH_{M_N})$, where $v$ is the homology class as given in (\ref{eqn:defn-homology}) and $N \geq 2$, if $h_1(e_k) = h_2(e_k)$ for all $1 \leq k \leq N-1$ then $h_1 = h_2$. 
\end{lem}
\begin{proof}
Observe that $h_1^{-1} \circ h_2$ acts as the identity on $\Z\{u, w\}^\perp \leq w^\perp$ and so $h_1^{-1} \circ h_2$ restricts to an automorphism of $\Z\{u, w\}$ preserving $Q$. The only automorphism of $(\Z\{u, w\}, Q|_{\Z\{u,w\}})$ fixing $w$ is the identity. Therefore, $h_1^{-1} \circ h_2 = \Id$ on $L$. In the case of $v \in H_2(M_N; \Z)$ for any $N \geq 2$, apply the above argument with $w = v$, $u = s$ and $\Z\{u, w\}^\perp = \Z\{e_1, \dots, e_{N-1}\}$. 
\end{proof}

Let $\Lambda_w$ denote the kernel of the natural map $h_w: \Stab(w) \to \OO(w^\perp/\Z\{w\}, \overline Q)$ (cf. Definition \ref{defn:lambda}). In order to describe $\Lambda_w$, we introduce an important type of element of $\OO(\HH_{M_N})$ used throughout this paper.
\begin{defn}
Let $N \geq 2$ and $u \in H_2(M_N; \Z)$ satisfy $Q_N(u,u) = \pm 1$ or $\pm 2$. The \emph{reflection} $\RRef_u$ about $u$ is an element of $\OO(\HH_{M_N})$ defined by
\[
	\RRef_u(x) = x - \frac{2Q_{M_N}(x,u)}{Q_{M_N}(u,u)}u.
\]
\end{defn}
In the lemma below, we use reflections and Eichler transformations to give generators for $\Lambda_w$.
\begin{lem}\label{lem:eichler}
Let $(L, Q)$ be any lattice and $w \in L$ be an isotropic vector. Let $A \leq w^\perp/\Z\{w\}$ denote the $\Z$-submodule of even elements with respect to $\overline Q$. Then there is an isomorphism of groups
\[
	E(w, \cdot): A \to \Lambda_w.
\]
In the case that $(L, Q) = \HH_{M_N}$ for any $N \geq 2$ and $w = v$, the group $\Lambda_w$ is generated by
\[
	f_k :=\RRef_{e_k}\circ \RRef_{e_{k+1}} \circ \RRef_{v-e_k-e_{k+1}} \circ \RRef_{e_k - e_{k+1}} 
\]
for $1 \leq k \leq N-2$ and $g := \RRef_{e_1} \circ \RRef_{v-e_1}$.
\end{lem}
\begin{proof}
For any $f \in \Lambda_w$, there exists $c(f) \in w^\perp/\Z\{w\}$ such that for any $e \in w^\perp$,
\[
	f(e) = e - \overline Q(c(f), e)w
\]
by definition of $\Lambda_w$ and unimodularity of $\overline Q$. This defines a homomorphism $c: \Lambda_w \to w^\perp/\Z\{w\}$ which is injective by Lemma \ref{lem:stab-v-determined}. 

For any $f \in \Lambda_w$, there exists $a, b \in \Z$ and $e \in \Z\{u, w\}^\perp$ such that
\[
	f(u) = u + aw + be
\]
because $Q(f(u), w) = 1$. Moreover, 
\[
	Q(u,u) = Q(f(u),\, f(u)) = Q(u,u)+  2a + b^2 Q(e, e)
\]
and so both $b^2Q(e,e)$ and $bQ(e,e)$ must be even. Because $c(f), e \in \Z\{w, u\}^\perp$, 
\begin{align*}
	0 &= Q(f(u),\, f(c(f))) = Q(f(u),\, c(f) - Q(c(f),c(f))w) =  bQ(e,\,c(f)) - Q(c(f),\,c(f)),\\
	0 &= Q(f(u), \, f(e)) = Q(f(u),\, e - Q(e, c(f))w) = bQ(e, \,e) - Q(e,\, c(f)).
\end{align*}
By the second equation, $Q(e, c(f))$ is even and by the first equation, $Q(c(f),c(f))$ is even. Hence $c(\Lambda_w) \leq A$.

Consider the homomorphism $E(w, \cdot): A \to \Lambda_w$ defined by
\[
	E(w, e): x \mapsto x + Q(w, x)e - Q(e, x)w - \frac 12 Q(e,e) Q(w, x) w 
\]
for each $e \in A$, where $E(w, e)$ is an \emph{Eichler transformation}. A computation shows that $c \circ E(w, \cdot) = \Id|_A$. Finally, if $(L, Q) = (H_2(M_N; \Z), Q_{M_N})$ and $w = v$, compute that $f_k = E(w, e_k + e_{k+1})$ for each $1 \leq k \leq N-2$ and $g = E(w, 2e_1)$, which together generate $A$.
\end{proof}

We combine the results of this subsection and record an important algebraic property of $\Stab(w)$. 
\begin{lem}\label{lem:ses}
For any primitive, isotropic vector $w \in L$, there is a split short exact sequence 
\[
	0 \to \Lambda_w \to \Stab(w) \xrightarrow{h_w} \OO(w^\perp/\Z\{w\}, \overline{Q})\to 0.
\]
In the case that $(L, Q) = \HH_{M_N}$ for any $N \geq 2$ and $w = f(v)$ for any $f \in \OO(\HH_{M_N})$, the split short exact sequence above is isomorphic to
\[
	0 \to\Z^{N-1} \to \Stab(w) \xrightarrow{h_w} \OO(N-1)(\Z) \to 0
\]
Therefore, $\Lambda_w \cong \Z^{N-1}$ is a finite-index maximal torsion-free subgroup of $\Stab(w)$ and a maximal torsion-free, abelian subgroup of $\OO(\HH_{M_N})$.
\end{lem}
\begin{proof}
There is a section $\ell$ of $h_w$ defined by
\[
	\ell: f \mapsto \Id \oplus f \in \OO(\Z\{u, w\} \oplus \Z\{u, w\}^\perp, Q) = \OO(L, Q)
\]
which shows that $h_w$ is surjective and the sequence is split. 

In the case of $(L, Q) =\HH_{M_N}$ with $N \geq 2$ and $w = f(v)$ for any $f \in \OO(\HH_{M_N})$, we can let $u = f(s)$, in which case 
\[
	(w^\perp/\Z\{w\}, \overline Q) \cong (\Z\{f(e_1), \dots, f(e_{N-1})\}, Q_{M_N})
\]
and so $\OO(w^\perp/\Z\{w\}, \overline Q) \cong \OO(N-1)(\Z)$ is finite. The subgroup 
$A \leq w^\perp/\Z\{w\}$ of even elements with respect to $\overline Q$ has index $2$ in $w^\perp/\Z\{w\}$ which has rank $N - 1$, and so $\Lambda_w \cong A \cong \Z^{N-1}$. Because the sequence is split, $\langle \Lambda_w, g \rangle$ must have torsion for any $g \in \Stab(w)$ with $g \notin \Lambda_w$ and so $\Lambda_w$ is a maximal torsion-free subgroup of $\Stab(w)$. 

For any $f \in \Stab(w)$, Lemma \ref{lem:infinite-order-parabolics} shows that $f$ is parabolic and $w \in H_2(M_N; \Z)$ is the unique element of $H_2(M_N; \Z)$ fixed by $f$, up to scaling. Suppose $g \in \OO(\HH_{M_N})$ commutes with some $f \in \Lambda_w$. Then $g(w) = \pm w$ because $f$ fixes $g(w)$, so $g \in \Stab(w)$ or $-g \in \Stab(w)$. If $\langle g, \Lambda_w \rangle$ is torsion-free then $g \in \Lambda_w$ or $-g \in \Lambda_w$ respectively. However if $-g \in \Lambda_w$ then $-g \circ g^{-1} = -\Id$ is torsion and is in $\langle g, \Lambda_w\rangle$. Therefore $g \in \Lambda_w$ which means that $\Lambda_w$ is a maximal, torsion-free, abelian subgroup of $\OO(\HH_{M_N})$.
\end{proof}
To use Lemma \ref{lem:ses}, we apply a theorem of Li--Li \cite[Theorem 4.2]{li--li} which says that for any $N \geq 2$ and any primitive, isotropic class $w \in H_2(M_N; \Z)$ of minimal genus $0$, there exists $\varphi \in \Diff^+(M_N)$ such that $[\varphi](v) = w$. Moreover, following elementary lemma strengthens this theorem in the case $2 \leq N \leq 8$. Recall the fixed $\Z$-basis $\{s, v, e_1, \dots, e_{N-1}\}$ of $H_2(M_N; \Z)$ given in (\ref{eqn:defn-homology}). 
\begin{lem}\label{lem:isotropic}
If $2 \leq N \leq 8$ and $w \in H_2(M_N; \Z)$ is an isotropic class, then
\begin{enumerate}[(a)]
\item there exists $f \in \OO(\HH_{M_N})$ such that $f(w) = v$ if $w$ is primitive, and \label{lem:isotropic-orbits}
\item $w$ has minimal genus $0$. \label{lem:minimal-genus}
\end{enumerate}
\end{lem}
\begin{proof}
The restriction of $Q_{M_N}$ to $\Z\{w, u\}$ is unimodular and indefinite so $\Z\{w, u\}^\perp$ is negative definite of rank $N-1 < 8$. By \cite{mordell}, $\Z\{w, u\}^\perp$ is isometric to $\Z\{e_1, \dots, e_{N-1}\}$; let $w_0 \in \Z\{w, u\}^\perp$ satisfy $Q_{M_N}(w_0,w_0) = -1$. 

With $a := Q_{M_N}(u,u)$, we have $Q_{M_N}(w, u-aw_0) = 1$ and 
\[
	Q_{M_N}(u-aw_0, u-aw_0) = a - a^2 \equiv 0 \pmod 2. 
\]
So $\Z\{w, u-aw_0\}$ is unimodular, even, and indefinite. By \cite{mordell} again, $\Z\{w, u-aw_0\}^\perp$ is negative definite and diagonal of rank $N-1$. There exists $f \in \OO(\HH_{M_N})$ that preserves the orthogonal direct sums below
\[
	f: \Z\{v, s\} \oplus \Z\{e_1, \dots, e_{N-1}\} \to \Z\{w, w_0-aw_1\} \oplus \Z\{w, w_0-aw_1\}^\perp
\]
such that $f(v) = w$. This proves \ref{lem:isotropic-orbits}.

To prove \ref{lem:minimal-genus}, we may assume that $w \neq 0$. Suppose $w_1\in H_2(M_N; \Z)$ is a primitive isotropic class such that $aw_1 = w$ for some $a \in \Z$. By \ref{lem:isotropic-orbits}, there exists some $f \in \Mod(M_N)$ such that $f(w_1) = v$. Because $N \leq 9$, there exists a diffeomorphism $\varphi \in \Diff^+(M_N)$ such that $[\varphi] = f$ by \cite[Theorem 2]{wall-diffeos}, and so the minimal genus of $w$ and the minimal genus of $av$ are equal, and the minimal genus of $av=a(H-E_1)$ is $0$ (cf. \cite[Theorem 4.2]{li--li}).
\end{proof}

\subsection{Lefschetz fibrations, conic bundles, and de Jonqui\'eres involutions}\label{sec:dejonq}
Let $N = 2m+1\geq 3$ be odd and fix some distinct complex numbers $a_1, \dots, a_{2m} \in \C$. Consider the birational map $\dJ_0: \CP^1 \times \CP^1 \dashrightarrow \CP^1 \times \CP^1$ given by
\[
	\dJ_0 : ([X_1:X_2],\, [Y_1:Y_2]) \mapsto \left([X_1:X_2], \,\left[ Y_2 \prod_{i=m+1}^{2m} (X_1 - a_i X_2) : Y_1 \prod_{i=1}^m (X_1 - a_i X_2)\right] \right). 
\]
Then $\dJ_0$ lifts to an automorphism $\dJ$ of order $2$ called a \emph{de Jonqui\'ere involution} of $X:= \Bl_P(\CP^1\times\CP^1)$ where
\[
	P := \{([a_i:1], \, [1:0]) : 1 \leq i \leq m \} \cup \{([a_i:1], \, [0:1] : m+1 \leq i \leq 2m)\}
\]
is a set of $2m$-many points in $\CP^1\times \CP^1$. Note that $X$ is diffeomorphic to $M_N$. Under this identification, $e_k \in H_2(M_N; \Z)$ is the class of the exceptional fiber above $([a_k:1], [1:0])$ for each $1 \leq k \leq m$ and the class of the exceptional fiber above $([a_k:1], [0:1])$ for each $m+1 \leq k \leq m$. 

Consider the projection map $\proj_0: \CP^1 \times \CP^1 \to \CP^1$ onto the first coordinate; it extends to a map $\proj: X \to \CP^1$ defining a holomorphic genus-$0$ Lefschetz fibration (in other words, a conic bundle). By construction, $\proj \circ \dJ = \proj$.

\begin{figure}
\centering
\includegraphics{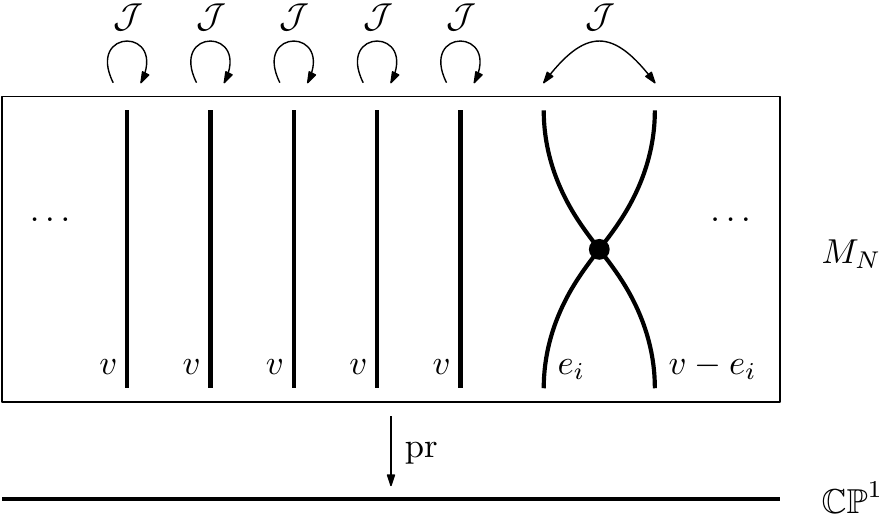}
\caption{Each line represents a copy of $\CP^1$ and is labeled with its homology class in $M_N$. The rightmost fiber, for $1 \leq i \leq N-1$, is a singular fiber. Each singular fiber is a union of two $(-1)$-spheres intersecting transversely once.}\label{fig:dejonq}
\end{figure}

If $z \neq z_k := [a_k:1] \in \CP^1$ for any $k$, the fiber of $\proj$ over a point $z \in \CP^1$ is $\{z\}\times \CP^1$ which is in the homology class $v \in H_2(M_N; \Z)$. Because $\dJ$ acts on each such $\proj^{-1}(z)$ in an orientation-preserving way, $[\dJ] \in \Stab(v) \leq \Mod(M_N)$. Moreover for all $1 \leq k \leq 2m$ and all $([a_k:1], [Y_1:Y_2]) \notin P$,
\[
	\dJ_0: ([a_k:1], [Y_1:Y_2]) \mapsto\begin{cases} 
	([a_k:1], [1:0]) & \text{if } 1 \leq k \leq m,\\
	([a_k:1], [0:1]) & \text{if }m+1 \leq k \leq 2m.
	\end{cases}
\]
Therefore, $[\dJ]$ must send the homology class $v - e_k$ of the strict transform of $\proj^{-1}([a_k:1])$ in $X$ to the exceptional divisor $e_k$. See Figure \ref{fig:dejonq} for an illustration of the action of $\dJ$ on the fibers of $\proj$. 

The maps $\proj$ and $\dJ$ described above will be used in the explicit constructions in Sections \ref{sec:parabolics} and \ref{sec:individual-stab-v}. The goal of the rest of this section is to show that it suffices to only consider the Lefschetz fibration $\proj:M_N \to \CP^1$ for our setting and to prove Proposition \ref{prop:no-lift-preserving-any-pi}. 
\begin{prop}\label{prop:no-other-lefschetz}
Let $p: M_N \to \Sigma$ be a Lefschetz fibration where $\Sigma$ is a closed, oriented surface and the generic fiber $F$ satisfies $[F] \neq 0 \in H_2(M_N; \Z)$. If $[F]$ has minimal genus $0$ then $\Sigma = \CP^1$ and $F = \CP^1$.
\end{prop}
\begin{proof}
Because $M_N$ is closed, a generic fiber $F$ is a compact submanifold of $M$ and has finitely many connected components. By \cite[Proposition 8.1.9]{gompf--stipsicz}, there is a bijection $\pi_1(\Sigma) \to \pi_0(F)$ because $\pi_1(M_N) = 0$. Therefore, $\Sigma = \CP^1$ and $F$ is connected since $\pi_0(F) = \pi_1(\CP^1) = 0$.

Because $[F]$ is nontrivial, $M_N$ can be given a symplectic structure such that $F$ is a symplectic submanifold (Gompf \cite[Theorem 10.2.18]{gompf--stipsicz}, \cite[Theorem 1.2]{gompf2}) and so $F$ must achieve the minimal genus in its homology class by the symplectic Thom conjecture (Oszv\'ath--Szab\'o \cite[Theorem 1.1]{oszvath--szabo}). 
\end{proof}

\begin{proof}[{Proof of Proposition \ref{prop:no-lift-preserving-any-pi}}]
Suppose there exists such a Lefschetz fibration $p: M_N \to \Sigma$ and a homeomorphism $h : \Sigma \to \Sigma$ with $p \circ \varphi = h \circ p$. Proposition \ref{prop:no-other-lefschetz} says that that $\Sigma = \CP^1$ and the generic fiber of $p$ has genus $0$. After blowing down the $(-1)$-spheres contained in the fibers of $p$, we see that $p$ must be a $\CP^1$-bundle over $\Sigma$ by \cite[Proposition 8.1.7]{gompf--stipsicz}. Because all $\CP^1$-bundles over $\CP^1$ are holomorphic, $M_N$ gets a complex structure as a rational surface and $p$ is holomorphic. 

We prove by induction on $N$ that if some homeomorphism $\varphi \in \Homeo^+(M_N)$ preserves a genus-$0$ Lefschetz fibration $p: M_N \to \CP^1$ then $[\varphi] \in \Mod(M_N)$ has finite order. If $N = 1$ then $\Mod(M_N)$ is finite. Now assume for some $N_0 > 1$ that the claim holds for any $1 \leq N < N_0$.

Let $N = N_0$ and suppose $\varphi \in \Homeo^+(M_N)$ preserves a genus-$0$ Lefschetz fibration $p: M_N \to \CP^1$. Then $\varphi$ must permute the singular fibers because none of the singular fibers are homeomorphic to a generic fiber $\CP^1$. There are finitely many singular fibers, so some power $\varphi^k$ must preserve each singular fiber. Each singular fiber $F$ of $p$ is a union of finitely many spheres of negative self-intersection intersecting transversely at finitely many points $q_1, \dots, q_m$. Because $\varphi^k$ restricts to a homeomorphism of each singular fiber, $\varphi^k$ must permute the points $q_1, \dots, q_m$. Moreover, $\varphi^k$ also restricts to a homeomorphism on $F - \{q_1, \dots, q_m\}$, a disjoint union of finitely many spheres with punctures. Therefore, a further power $\varphi^{k\ell}$ must preserve each component of $F - \{q_1, \dots, q_m\}$ and its orientation.

Let $S \subseteq F$ be an embedded $(-1)$-sphere in $M_N$ which only intersects one other sphere $S_0 \subseteq F$ of negative self-intersection, at the point $q_m$. Because $\varphi^{k\ell}$ fixes $q_m$ and preserves $S - \{q_m\} \subseteq F - \{q_1, \dots, q_m\}$, the homeomorphism $\varphi^{k\ell}$ must preserve $S \subseteq M_N$. Let $b: M_N \to M$ be the map that blows down $S$ to a point $q \in M$. Because $M$ is a rational surface, $M$ is diffeomorphic to $M_{N-1}$ or $\CP^1\times \CP^1$. 

Because $\varphi^{k\ell}$ defines a homeomorphism on $M_N - S$, it induces a homeomorphism of $M - q$ that extends to a homeomorphism $\psi$ of $M$ and preserves the Lefschetz fibration $p': M \to \CP^1$ such that $p = p' \circ b$. If $M \cong M_{N-1}$ then $[\psi]$ has finite order in $\Mod(M)$ by the inductive hypothesis. Otherwise, $M \cong \CP^1\times \CP^1$ and so $\Mod(M)$ is finite. Therefore, $[\psi]$ also has finite order in $\Mod(M)$. 

Finally, note that $b_*: H_2(M_N; \Z) \to H_2(M; \Z)$ induces the quotient map
\[
	H_2(M_N; \Z) \cong \Z\{[S]\}^\perp \oplus \Z\{[S]\} \to \Z\{[S]\}^\perp \cong H_2(M;\Z).
\]
Because $\psi = b \circ \varphi^{k\ell}$ and $\varphi_*^{k\ell}([S]) = [S]$, the restriction of $\varphi_*^{k\ell}$ to $\Z\{[S]\}^\perp$ must have the same order as $\psi_*$. Finally, this shows that $[\varphi^{k\ell}]$, and therefore $[\varphi]$, has finite order in $\Mod(M_N)$. 
\end{proof}

\section{Theorem \ref{thm:parabolics}: Lifting $\Lambda_w$ to $\Diff^+(M_N)$}\label{sec:thm-parabolics}\label{sec:parabolics}
This section is dedicated to the proof of Theorem \ref{thm:parabolics}. Before proceeding with the proof, we fix notation regarding certain subsets of $\CP^1$ illustrated in Figure \ref{fig:sets-in-p1}. Let $N = n+1$ and $m = \lceil \frac n2 \rceil$ so that if $n$ is even, $n+1 = 2m+1$ and if $n$ is odd, $n = 2m+1$. Fix distinct complex numbers $a_1, \dots, a_{2m} \in \C$ and let $z_k := [a_k : 1]$ for all $k = 1, \dots, n$. Then
\begin{enumerate}[(a)]
\item for each $1 \leq k \leq n-1$, let $B_k \cong \mathbb D^2$ denote a closed disk in $\CP^1 - \{[1:0]\}$ containing $z_k$ and $z_{k+1}$ and no other points $z_j$ for $j \neq k, k+1$ so that $B_k \cap B_{k'} = \emptyset$ if $\lvert k - k' \rvert > 1$, 
\item for each $1 \leq k \leq n-1$, let $U_k \cong [0, 1] \times \sphere^1 \subseteq B_k$ denote a collar neighborhood of $B_k$ where $\{0\} \times \sphere^1$ corresponds to $\partial B_k$, and
\item for each $1 \leq k \leq n$, let $V_k\cong \mathbb D^2$ denote a closed disk in $B_k$ if $k = 1$ or in $(B_k-U_k) \cap (B_{k-1}-U_{k-1})$ if $k \geq 2$ with $z_k \in V_k$.
\end{enumerate}

\begin{figure}
\centering
\includegraphics[width=0.65\textwidth]{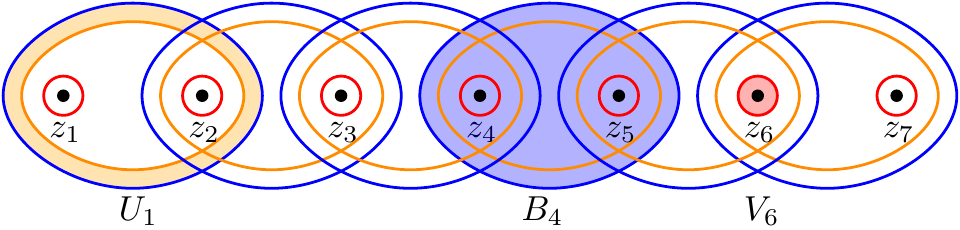}
\caption{This figure depicts the sets $U_k$, $B_k$, and $V_k$ in the case $n = 7$. In particular, $B_4$ is shaded in blue and surrounds $z_4$ and $z_5$. The annulus $U_1$, shaded in orange, is the collar neighborhood of $B_1$ surrounding $z_1$ and $z_2$. The disk $V_6$, shaded in red, contains $z_6$ and contained in $(B_5-U_5) \cap (B_6-U_6)$.} \label{fig:sets-in-p1}
\end{figure}

As in Section \ref{sec:dejonq}, let
\[
	P = \{([a_i:1], \, [1:0]) : 1 \leq i \leq m \} \cup \{([a_i:1], \, [0:1]) : m+1 \leq i \leq 2m\}
\]
and consider the de Jonqui\'ere involution $\dJ$ on $\Bl_P(\CP^1 \times \CP^1)$. Identify $M_{n+1}$ with
\begin{enumerate}[(a)]
	\item $\Bl_P(\CP^1 \times \CP^1)$ if $n$ is even, and
	\item $\Bl_{P-\{([a_{2m} : 1], \, [0:1])\}}(\CP^1\times \CP^1)$ if $n$ is odd.
\end{enumerate}
In both cases, consider $\proj: M_{n+1} \to \CP^1$ defined in Section \ref{sec:dejonq}. There is a natural inclusion
\[
	\proj^{-1}(B_k) \hookrightarrow \Bl_P(\CP^1\times \CP^1)
\]
that is preserved by $\dJ$ on $\Bl_P(\CP^1\times \CP^1)$ for all $1 \leq k \leq n-1$. We use this inclusion to define $\dJ|_{\proj^{-1}(B_k)}$ on each $\proj^{-1}(B_k) \subseteq M_{n+1}$ regardless of the parity of $n$. Note that $\dJ|_{\proj^{-1}(B_k)} = \dJ|_{\proj^{-1}(B_{k+1})}$ when restricted to $\proj^{-1}(B_k) \cap \proj^{-1}(B_{k+1})$ for all $1\leq k \leq n-2$.

There are four main steps to the proof of Theorem \ref{thm:parabolics}.
\begin{enumerate}[(1)]
\item Construct commuting diffeomorphisms $\gamma_1, \dots, \gamma_{n-1} \in \Diff^+(M_N)$ that preserve the genus-$0$ holomorphic Lefschetz fibration $\proj: M_N \to \CP^1$ such that $\supp(\gamma_k) \subseteq \proj^{-1}(B_k)$ and $\gamma_k$ agrees with $\dJ$ on $\proj^{-1}(B_k - U_k)$ for each $1 \leq k \leq n-1$. These maps should be thought of as \emph{local de Jonqui\'eres maps}.
\item Construct commuting diffeomorphisms $r_1, \dots, r_n \in \Diff^+(M_N)$ with $\supp(r_k) \subseteq \proj^{-1}(V_k)$ so that $[r_k] = \RRef_{e_k}$ for each $1 \leq k \leq n$.
\item Define a homomorphism $\rho_v: \Lambda_v \to \Diff^+(M_N)$ using the diffeomorphisms above so that $\rho_v$ is a section of $q: \Diff^+(M_N) \to \Mod(M_N)$ and $\rho_v(\Lambda_v)$ almost preserves $\proj: M_N \to \CP^1$.
\item Define a homomorphism $\rho_w: \Lambda_w \to \Diff^+(M_N)$ for any other primitive, isotropic class $w$ of minimal genus $0$ by pre- and post-composing $\rho_v$ by conjugation in $\Mod(M_N)$ and $\Diff^+(M_N)$.
\end{enumerate}

\subsection{Step 1: Constructing local de Jonqui\'eres maps $\gamma_1, \dots, \gamma_{n-1}\in \Diff^+(M_N)$}
For each $1\leq k \leq 2m$, define $\lambda_k: U_k \to \C^\times$ by 
\[
	\lambda_k(x) := \sqrt{\frac{\prod_{i=1}^m (x-z_i)}{\prod_{i=m+1}^n (x-z_i)}}
\]
with any smooth choice of square root. Such a choice is well-defined because $U_k$ is an annulus surrounding but not containing two points $z_k$ and $z_{k+1}$. Moreover, $x - z_i \neq 0$ because $z_i \notin U_k$ for any $1 \leq i \leq n$. For such a choice of $\lambda_k$, consider the map $M_{\lambda_k}: U_k \to \PGL_2(\C)$ given by 
\[
	M_{\lambda_k}(x) = \begin{pmatrix}
	1 & 1 \\ -\lambda_k(x) & \lambda_k(x)
	\end{pmatrix}\in \PGL_2(\C).
\]
We also record the inverse of $M_{\lambda_k}(x)$ for later use:
\[
	M_{\lambda_k}(x)^{-1} = \begin{pmatrix}
	1 & - \frac{1}{\lambda_k(x)} \\
	1 & \frac{1}{\lambda_k(x)}
	\end{pmatrix} \in \PGL_2(\C).
\]
Viewing $M_{\lambda_k}(x)$ and $M_{\lambda_k}(x)^{-1}$ as automorphisms of $\CP^1$, define a diffeomorphism $u_{\lambda_k}$ of $\proj^{-1}(U_k) = U_k \times \CP^1$ by
\[
	u_{\lambda_k}([x:1], [Y_1:Y_2]) = ([x:1], \, M_{\lambda_k}(x) \cdot [Y_1:Y_2]).
\]

Let $T: [0,1] \to [0, 1]$ be a smooth, nondecreasing function such that $T|_{[0, \varepsilon]} \equiv 0$ and $T|_{[1-\varepsilon, 1]} \equiv 1$ for some $0 < \varepsilon \ll 1$. Identifying $\proj^{-1}(U_k) = U_k \times \CP^1$ with $[0, 1] \times \partial B_k \times \CP^1$ (cf. Figure \ref{fig:sets-in-p1}), let $j_k$ be a diffeomorphism of $\proj^{-1}(U_k)$ defined by
\[
	j_k(t,\, \theta,\, [Y_1:Y_2] ) = (t, \, \theta, \, [e^{\sqrt{-1} \pi T(t)} Y_1 : Y_2]). 
\]
Roughly, $j_k$ is a map on $[0, 1] \times \sphere^1 \times \CP^1$ induced by an isotopy of $\sphere^1\times \CP^1$ from the $\Id \times \Id$ to $\Id \times R(\pi)$, where $R(\pi)$ is a rotation-by-$\pi$ map on $\CP^1$. See Figure \ref{fig:unraveling}.

\begin{figure}
\centering
\includegraphics[width=\textwidth]{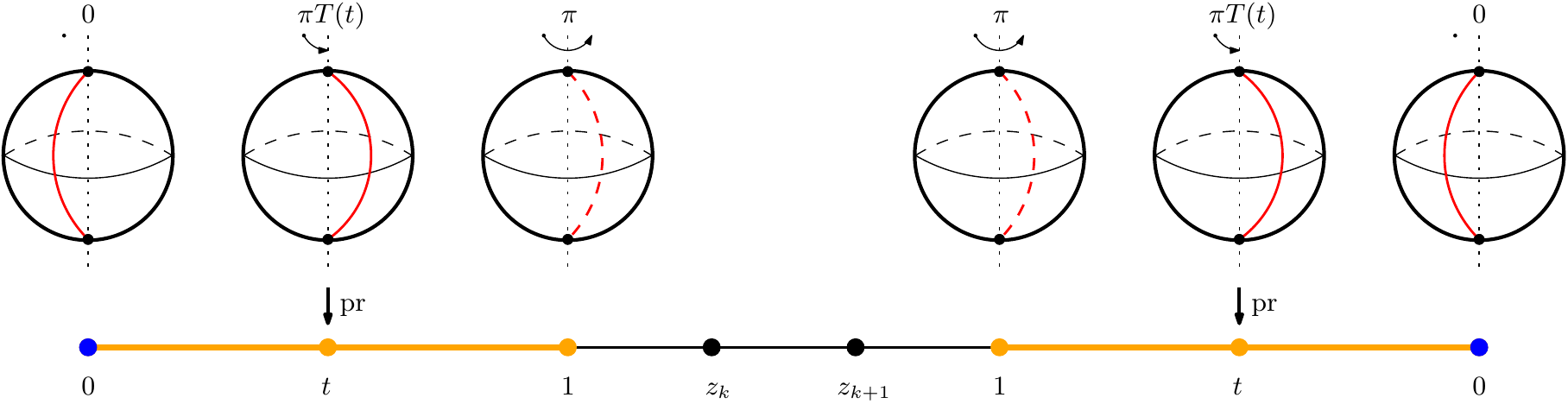}
\caption{This figure illustrates the action of $j_k$ on $\proj^{-1}(U_k) \cong [0, 1] \times \partial B_k \times \CP^1$. The horizontal line represents $B_k \subseteq \CP^1$ and the orange portion represents the annulus $U_k \subseteq B_k \subseteq \CP^1$ whose width is parametrized by $t \in [0, 1]$. The blue points represent $\partial B_k \cong \sphere^1$. The diffeomorphism $j_k$ acts by rotation-by-$\pi T(t)$ on the spheres lying above a point $(t, x) \in [0, 1 ]\times \partial B_k \cong U_k$.}\label{fig:unraveling}
\end{figure}
In the next lemma, we use the fact that the de Jonqui\'eres map is conjugate to $\Id\times R(\pi)$ on each $\{t\} \times \partial B_k \times \CP^1$ to modify $\dJ|_{\proj^{-1}(U_k)}$ to be the identity near the boundary $\proj^{-1}(\partial B_k)$. 
\begin{lem}\label{lem:diffeo-bdry}
Let $1 \leq k \leq n-1$. On $\proj^{-1}(U_k) = [0, 1] \times \partial B_k \times \CP^1$,
\[
	u_{\lambda_k} \circ j_k \circ u_{\lambda_k}^{-1} = \begin{cases}
	\dJ & \text{ on }\proj^{-1}([1-\varepsilon, 1] \times \partial B_k), \\
	\Id & \text{ on }\proj^{-1}([0, \varepsilon] \times \partial B_k).
	\end{cases}
\]
\end{lem}
\begin{proof}
On $\proj^{-1}([0, \varepsilon] \times \partial B_k)$, note that $j_k \equiv \Id$. On $\proj^{-1}([1-\varepsilon, 1] \times \partial B_k)$, 
\[
	j_k([x:1], [Y_1:Y_2]) = ([x:1], [-Y_1:Y_2]).
\]
For all $[x:1] \in [1-\varepsilon, 1] \times \partial B_k \subseteq U_k$ 
\[
	M_{\lambda_k}(x) \begin{pmatrix}
	-1 & 0 \\
	0 & 1
	\end{pmatrix} M_{\lambda_k}(x)^{-1} =  \begin{pmatrix}
	0 & 1 \\ \lambda_k(x)^2 & 0
	\end{pmatrix} \in \PGL_2(\C),
\]
and so 
\begin{align*}
u_{\lambda_k} \circ j_k \circ u_{\lambda_k}^{-1}([x:1], [Y_1:Y_2]) &= \left([x:1], [Y_2 : \lambda_k(x)^2 Y_1]\right) = \dJ([x:1], [Y_1:Y_2]).\qedhere
\end{align*}
\end{proof}

The diffeomorphisms $\gamma_k$ below should be thought of as \emph{local} de Jonqui\'eres maps, acting only on a single pair of singular fibers of $\proj$. 
\begin{defn}
For any $1 \leq k \leq n-1$, let $\gamma_k$ be the diffeomorphism
\[
	\gamma_k = \begin{cases}
	\dJ &\text{ on }\proj^{-1}(B_k - U_k) \\
	u_{\lambda_k} \circ j_k \circ u_{\lambda_k}^{-1} &\text{ on }\proj^{-1}(U_k) \\
	\Id &\text{ on }\proj^{-1}(\CP^1-B_k).
	\end{cases}
\]
\end{defn}

\begin{prop}\label{prop:diffeo-properties}
The diffeomorphisms $\gamma_k$ satisfy the following properties: 
\begin{enumerate}[(a)]
	\item The diffeomorphism $\gamma_k$ preserves $\proj$ for all $1 \leq k \leq n-1$. In particular, $\proj \circ \gamma_k = \proj$.\label{prop:diffeo-property-1}
	\item The diffeomorphisms $\gamma_i$ and $\gamma_j$ commute for all $1 \leq i, j\leq n-1$. \label{prop:diffeo-property-2}
	\item As mapping classes, $[\gamma_k] = \RRef_{v-e_k-e_{k+1}} \circ \RRef_{e_k - e_{k+1}}$ for all $1 \leq k \leq n-1$. \label{prop:diffeo-property-3}
\end{enumerate}
\end{prop}
\begin{proof}
For each $k$, $\proj\circ u_{\lambda_k} = \proj$ and $\proj \circ j_k = \proj$ by construction of $u_{\lambda_k}$ and $j_k$ when restricted to $\proj^{-1}(U_k)$. Therefore,
\[
	(\proj \circ \gamma_k)|_{\proj^{-1}(U_k)} = \left(\proj \circ (u_{\lambda_k} \circ j_k \circ u_{\lambda_k}^{-1}))\right|_{\proj^{-1}(U_k)} = \proj|_{\proj^{-1}(U_k)}
\]
and $\gamma_k$ preserves the fibers of $\proj$ on $\proj^{-1}(U_k)$ for all $k$. The same is obviously true on $\proj^{-1}(\CP^1 - B_k)$ and true on $\proj^{-1}(B_k - U_k)$ by construction of $\dJ$. This proves \ref{prop:diffeo-property-1}.

If $\lvert i-j \rvert > 1$ then $\supp(\gamma_i) \cap \supp(\gamma_j) = \emptyset$ so $\gamma_i$ and $\gamma_j$ commute. To show that $\gamma_i$ and $\gamma_{i+1}$ commute, we will consider the action of these two diffeomorphisms on 
\[
	\proj^{-1}(B_i \cap B_{i+1}) = \proj^{-1}(\underbrace{(B_i \cap B_{i+1}) \cap (U_i \cup U_{i+1})}_{=: \mathcal C_{i}}) \cup \proj^{-1}(\underbrace{(B_i \cap B_{i+1}) - (U_i \cup U_{i+1})}_{=: \mathcal V_{i}})
\]
which contains $\supp(\gamma_i) \cap \supp(\gamma_{i+1})$. See Figure \ref{fig:commuting}.

\begin{figure}
\centering
\includegraphics[width=0.25\textwidth]{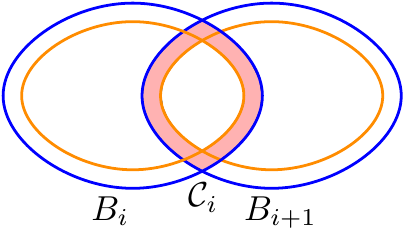}
$\quad\quad$
\includegraphics[width=0.25\textwidth]{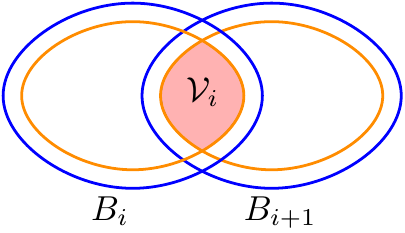}
\caption{For each $1 \leq i \leq n-1$, the sets $\mathcal C_i$ (left) and $\mathcal V_i$ (right) are contained in $B_i \cap B_{i+1}$.}\label{fig:commuting}
\end{figure}

By construction,
\[
	\gamma_i|_{\proj^{-1}(\mathcal V_{i})}= \dJ|_{\proj^{-1}(\mathcal V_{i})} = \gamma_{i+1}|_{\proj^{-1}(\mathcal V_{i})} 
\]
and so $\gamma_i$ and $\gamma_{i+1}$ commute on $\proj^{-1}(\mathcal V_{i})$.

For any $[x:1] \in \mathcal C_i$, both $\gamma_i$ and $\gamma_{i+1}$ act on $\proj^{-1}([x:1])$ by \ref{prop:diffeo-property-1}. If $[x:1] \in U_i \cap U_{i+1}$ then for some $t, T \in [0, 1]$ depending on $x$,
\begin{align*}
\gamma_i([x:1], [Y_1:Y_2]) &= \left([x:1], \,\left(M_{\lambda_i}(x) \begin{pmatrix} e^{\sqrt{-1}\pi t} & 0 \\ 0 & 1
\end{pmatrix} M_{\lambda_i}(x)^{-1}\right) \cdot [Y_1:Y_2]\right) \\
\gamma_{i+1}([x:1], [Y_1:Y_2]) &= \left([x:1], \,\left(M_{\lambda_{i+1}}(x) \begin{pmatrix} e^{\sqrt{-1}\pi T} & 0 \\ 0 & 1
\end{pmatrix} M_{\lambda_{i+1}}(x)^{-1}\right) \cdot [Y_1:Y_2]\right)
\end{align*}
Moreover, $\lambda_i(x) = \lambda_{i+1}(x)$ or $-\lambda_{i+1}(x)$. In the first case, $M_{\lambda_i}(x) = M_{\lambda_{i+1}}(x)$ so $\gamma_i$ and $\gamma_{i+1}$ commute on $\proj^{-1}([x:1])$. In the second case, compute for each $[x:1] \in U_i$ that 
\[
	M_{\lambda_i}(x) = M_{-\lambda_i}(x) \begin{pmatrix}
	0 & 1 \\ 1 & 0
	\end{pmatrix} \qquad \text{ and }\qquad M_{\lambda_i}(x)^{-1} = \begin{pmatrix}
	0 & 1 \\ 1 & 0
	\end{pmatrix} M_{-\lambda_i}(x)^{-1}
\]
and so
\[
	M_{-\lambda_i}(x) \begin{pmatrix} 
	e^{\sqrt{-1}\pi T} & 0 \\ 0 & 1 
	\end{pmatrix} M_{-\lambda_i}(x)^{-1} = M_{\lambda_i}(x) \begin{pmatrix} 
	1 & 0 \\ 0 & e^{\sqrt{-1}\pi T}
	\end{pmatrix} M_{\lambda_i}(x)^{-1}.
\]
It is clear that
\[
	M_{\lambda_i}(x) \begin{pmatrix}
	e^{\sqrt{-1}\pi t} & 0 \\
	0 & 1
	\end{pmatrix} M_{\lambda_i}(x)^{-1}  \qquad \text{ and } \qquad M_{\lambda_i}(x) \begin{pmatrix} 
	1 & 0 \\ 0 & e^{\sqrt{-1}\pi T}
	\end{pmatrix} M_{\lambda_i}(x)^{-1}
\]
commute in $\PGL_2(\C)$, which shows that $\gamma_i$ and $\gamma_{i+1}$ commute on $\proj^{-1}([x:1])$ in this case. 

If $[x: 1] \in U_i$ and $[x:1] \notin U_{i+1}$ then for some $t \in [0, 1]$ depending on $x$ and for all $[Y_1:Y_2] \in \CP^1$,
\begin{align*}
	\gamma_i([x:1], [Y_1:Y_2]) &= \left([x:1], \,\left(M_{\lambda_i}(x) \begin{pmatrix} e^{\sqrt{-1} \pi  t} & 0 \\ 0 & 1
	\end{pmatrix}  M_{\lambda_i}(x)^{-1}\right) \cdot [Y_1:Y_2]\right) \\
	\gamma_{i+1}([x:1], [Y_1:Y_2]) &= \dJ([x:1], [Y_1:Y_2]) = \left([x:1], \,\left(M_{\lambda_i}(x) \begin{pmatrix} -1 & 0 \\ 0 & 1
	\end{pmatrix}  M_{\lambda_i}(x)^{-1}\right) \cdot [Y_1:Y_2]\right)
\end{align*}
where the second equality follows from (the proof of) Lemma \ref{lem:diffeo-bdry}. Therefore, $\gamma_i$ and $\gamma_{i+1}$ commute on $\proj^{-1}([x:1])$. By analogous computations, $\gamma_i$ and $\gamma_{i+1}$ commute on $\proj^{-1}([x:1])$ if $x \in U_{i+1}$ and $x \notin U_i$. This proves \ref{prop:diffeo-property-2}.

Finally, note that for all $j \neq k, k+1$, the map $\gamma_k$ restricts to the identity on $e_j$ and on $\proj^{-1}(z)$ for any $z \notin B_k$ so $(\gamma_k)_*(e_j) = e_j$ and $(\gamma_k)_*(v) = v$. Moreover, $\gamma_k$ agrees with $\dJ$ on $\proj^{-1}(B_k)$, meaning that $(\gamma_k)_*(e_j) = v-e_j$ for $j = k$ and $j = k+1$. This then determines $[\gamma_k] \in \Mod(M_{n+1})$ by Lemma \ref{lem:stab-v-determined}. A computation shows that the same holds for $\RRef_{v-e_k-e_{k+1}} \circ \RRef_{e_k - e_{k+1}}$.
\end{proof}

\subsection{Step 2: Constructing $r_1, \dots, r_{n}\in \Diff^+(M_N)$}
For each $1 \leq k \leq n$, the exceptional divisor $e_k$ has a tubular neighborhood $\nu_k$ in $\proj^{-1}(V_k)$ that is diffeomorphic to $\overline{\CP^2} - \{[0:0:1]\}$. Let $i_k: \overline{\CP^2} - \{[0:0:1]\} \to \nu_k$ be this diffeomorphism and let $\tau_0$ be a diffeomorphism of $\overline{\CP^2} - \{[0:0:1]\}$ given by complex conjugation, $\tau_0: [X:Y:Z] \mapsto [\bar X: \bar Y : \bar Z]$. 

Consider a smooth path $\eta: (0, 1) \to \SO(4)$ such that
\[
	\eta(t) = \begin{cases}
	\textup{diag}(1, -1, 1, -1) & \text{ if }t \in (1, 1-\varepsilon) \\
	\Id & \text{ if }t \in (0, \varepsilon)
	\end{cases}
\]
for some $0 < \varepsilon \ll 1$. Let $B$ denote the punctured ball in $\overline{\CP^2} - \{[0:0:1]\}$ given by 
\[
	B := \{[a+b\sqrt{-1}:c+d\sqrt{-1}:1] \in \overline{\CP^2} : 0 < \lVert (a+b\sqrt{-1},c+d\sqrt{-1})\rVert < 1\} \cong (0, 1) \times \sphere^3 \subseteq \R^4, 
\]
and define $\tau \in \Diff^+(\overline{\CP^2}-\{[0:0:1]\})$ by 
\[
	\tau = \begin{cases}
	\tau_0 & \text{ on } \overline{\CP^2} - B, \\
	(t, x) \mapsto (t, \eta(t)x)  & \text{ on }B\cong (0, 1) \times \sphere^3.
	\end{cases}
\]
Then $\tau$ is compactly supported in $\overline{\CP^2} - \{[0:0:1]\}$.
\begin{defn}\label{defn:rk}
For all $1 \leq k \leq n$, let $r_k \in \Diff^+(M_{N})$ be
\[
	r_k := \begin{cases}
	i_k \circ \tau \circ i_k^{-1} &\text{ on }\nu_k \\
	\Id & \text{ on }M_{N} - \nu_k
	\end{cases}
\]
\end{defn}
\begin{rmk}\label{rmk:rk}
By construction, the diffeomorphism $r_k$ restricts to an orientation-reversing diffeomorphism of $e_k$ and preserves the homology classes $e_i$ for all $i \neq k$ and $v$. This forces $[r_k] = \RRef_{e_k} \in \Mod(M_{N})$ by Lemma \ref{lem:stab-v-determined}. Moreover, $\supp(r_k) \subseteq \nu_k \subseteq \proj^{-1}(V_k)$ and $r_k$ preserves $\nu_k \subseteq\proj^{-1}(V_k)$. 
\end{rmk}

\subsection{Step 3: Constructing $\rho_v: \Lambda_v \to \Diff^+(M_N)$}
The generators $f_1, \dots, f_{n-1}$ of $\Lambda_v$ (cf. Lemma \ref{lem:eichler}) will be mapped under $\rho_v$ to the following diffeomorphisms. 
\begin{lem}\label{lem:phi-commute}
For each $1 \leq k \leq n-1$, let
\[
	\varphi_k := r_k \circ r_{k+1} \circ \gamma_k.
\]
Then $\varphi_i \circ \varphi_j = \varphi_j \circ \varphi_i$ for any $1 \leq i, j \leq n$. 
\end{lem}
\begin{proof}
For any $1 \leq i, j \leq n$, the diffeomorphisms $\varphi_i$ and $\varphi_j$ commute if $\lvert i - j \rvert > 1$ because they have disjoint support. For any $1 \leq i \leq n-1$, 
\[
	\varphi_i|_{\proj^{-1}(V_{i+1})} = (r_{i+1} \circ \dJ)|_{\proj^{-1}(V_{i+1})} = \varphi_{i+1} |_{\proj^{-1}(V_{i+1})}
\]
so $\varphi_i$ and $\varphi_{i+1}$ commute on $\proj^{-1}(V_{i+1})$. Moreover on $S_i := \proj^{-1}(B_i \cap B_{i+1}) -\proj^{-1}(V_{i+1})$
\[
	\varphi_i|_{S_i} = \gamma_i, \qquad \varphi_{i+1} |_{S_i} = \gamma_{i+1}
\]
and so $\varphi_i$ and $\varphi_{i+1}$ commmute on $\proj^{-1}(B_i \cap B_{i+1})$ by Proposition \ref{prop:diffeo-properties}\ref{prop:diffeo-property-2}. Finally, $\varphi_i$ and $\varphi_{i+1}$ commute on $M_{n+1} - \proj^{-1}(B_i \cap B_{i+1})$ because $\supp(\varphi_i) \cap \supp(\varphi_{i+1})$ is contained in $\proj^{-1}(B_i \cap B_{i+1})$. 
\end{proof}
It remains to construct the image of the last generator $g$ of $\Lambda_v$ under $\rho_v$. 
\begin{lem}\label{lem:psi}
The following is a well-defined diffeomorphism:
\[
	\psi = \begin{cases}
	\varphi_1 \circ \varphi_1 & \text{ on }\proj^{-1}(V_1) \\
	\Id & \text{ on }M_{n+1} - \proj^{-1}(V_1).
	\end{cases}
\]
Moreover, 
\begin{enumerate}[(a)]
\item the map $\psi$ commutes with $\varphi_k$ for all $1 \leq k \leq n$ and \label{lem:psi-commute}
\item in $\Mod(M_{n+1})$, $[\psi] = g = \RRef_{e_1} \circ \RRef_{v-e_1}$. \label{lem:psi-mod}
\end{enumerate}
\end{lem}
\begin{proof}
By definition, $r_k$ has support contained in the interior of $\proj^{-1}(V_k)$ for all $1 \leq k \leq n$. So $r_1|_C = \Id|_C$ on some collar neighborhood $C$ of $\proj^{-1}(V_1)$, and
\[
	\psi|_C = (\varphi_1 \circ \varphi_1)|_{C} = (r_1 \circ \dJ \circ r_1 \circ \dJ))|_C = (\dJ \circ \dJ)|_C = \Id|_C.
\]
Moreover, $\dJ$, $r_1$, and $r_2$ all preserve $\proj^{-1}(V_1)$, so the map $\psi$ is indeed a diffeomorphism.

The diffeomorphisms $\psi$ and $\varphi_k$ have disjoint supports for all $k > 1$. Considering the subsets $\proj^{-1}(V_1)$ and $M_{n+1} - \proj^{-1}(V_1)$ separately shows that $\varphi_1$ and $\psi$ commute as well.

Compute for all $2 \leq k \leq n$ that $[\psi](e_k) = e_k$ because $\supp(\psi) \subseteq \proj^{-1}(V_1)$. Moreover, $\psi$ agrees with $\varphi_1^2$ on $\proj^{-1}(V_1)$, meaning that 
\[
	[\psi](e_1) = [\varphi_1^2](e_1) = e_1+2v.
\]
Computing that 
\[
	\RRef_{e_1} \circ \RRef_{v-e_1}(e_k) = [\psi](e_k)
\] 
for all $1 \leq k \leq n$ and applying Lemma \ref{lem:stab-v-determined} shows that $[\psi] = g$.
\end{proof}

\begin{prop}\label{prop:rho-v}
There is a homomorphism $\rho_v: \Lambda_v \to \Diff^+(M_{N})$ defined by
\[
	\rho_v(f_k) := \varphi_k \quad\text{ for all }1 \leq k \leq n-1, \qquad \rho_v(g) := \psi
\]
where $g$ and $f_k$ for $1 \leq k \leq n-1$ are the generators of $\Lambda_v$ as given in Lemma \ref{lem:eichler}. Moreover, 
\begin{enumerate}[(a)]
\item $\rho_v$ is a section of the map $q: \Diff^+(M_N) \to \Mod(M_N)$ restricted to $\Lambda_v \leq \Mod(M_N)$, and \label{prop:rho-v-1} 
\item for all $\varphi \in \rho_v(\Lambda_v)$,
\[
	(\proj\circ\varphi)|_{M_N - \bigcup_{i=1}^n\proj^{-1}(V_i)} = \proj|_{M_N - \bigcup_{i=1}^n\proj^{-1}(V_i)}.
\]
Hence $\rho_v(\Lambda_v)$ almost preserves the Lefschetz fibration $\proj: M_N \to \CP^1$. \label{prop:rho-v-2} 
\end{enumerate}
\end{prop}
\begin{proof}
By Lemma \ref{lem:eichler}, $\Lambda_v \cong \Z^n$ is generated by $f_1, \dots, f_{n-1}, g$. By Lemmas \ref{lem:phi-commute} and \ref{lem:psi}\ref{lem:psi-commute}, the image of $\rho_v$ is abelian and therefore $\rho_v$ is a well-defined homomorphism. 

Compute using Proposition \ref{prop:diffeo-properties}\ref{prop:diffeo-property-3} and Remark \ref{rmk:rk} that
\[
	[\rho_v(f_k)] = [r_k] \circ [r_{k+1}] \circ [\gamma_k] = \RRef_{e_k} \circ \RRef_{e_{k+1}} \circ \RRef_{v-e_k-e_{k+1}} \circ \RRef_{e_k-e_{k+1}} = f_k \in \Lambda_v.
\]
Lemma \ref{lem:psi} shows that $[\rho_v(g)] = g$. Therefore, $\rho_v$ is a section of the quotient map $q: \Diff^+(M_N) \to \Mod(M_N)$ restricted to $\Lambda_v \leq \Mod(M_N)$.

Finally, $\supp(r_k) \subseteq \proj^{-1}(V_k)$ for all $1 \leq k \leq n$ (cf. Remark \ref{rmk:rk}). By Proposition \ref{prop:diffeo-properties}\ref{prop:diffeo-property-1}, $\proj \circ\gamma_k = \proj$ for all $1 \leq k \leq n-1$, so
\[
	\varphi_k|_{M_{n+1} - \bigcup_{i=1}^n \proj^{-1}(V_i)} = \gamma_k|_{M_{n+1} - \bigcup_{i=1}^n \proj^{-1}(V_i)}.
\]
By Lemma \ref{lem:psi},
\[
	\psi|_{M_{n+1} - \bigcup_{i=1}^n \proj^{-1}(V_i)} = \Id|_{M_{n+1} - \bigcup_{i=1}^n \proj^{-1}(V_i)}. 
\]
Hence $\proj \circ \varphi = \proj$ restricted to $M_N - \bigcup_{i=1}^n \proj^{-1}(V_i)$ for all $\varphi \in \rho_v(\Lambda_v)$ and so $\rho_v(\Lambda_v)$ almost preserves $\proj$. 
\end{proof}

\subsection{Step 4: Extension to any primitive, isotropic class $w$ of minimal genus $0$}
With the constructions above in hand, we conclude the proof of Theorem \ref{thm:parabolics}.
\begin{proof}[Proof of Theorem \ref{thm:parabolics}]
Because $w$ is a primitive, isotropic class of minimal genus $0$, there exists some $\alpha \in \Diff^+(M_N)$ such that $[\alpha](w) = v$ by a theorem of Li--Li (\cite[Theorem 4.2]{li--li}). Using the definitions of $h_v$ and $h_w$ (cf. Lemma \ref{lem:ses}), compute that
\[
	\Lambda_v   = [\alpha]\circ \Lambda_w\circ [\alpha^{-1}]
\] 
and define $\rho_w: \Lambda_w \to \Diff^+(M_N)$ by 
\[
	\rho_w(f) = \alpha^{-1} \circ \rho_v([\alpha] \circ f \circ [\alpha^{-1}]) \circ \alpha
\]
where $\rho_v: \Lambda_v \to \Diff^+(M_N)$ is the homomorphism constructed in Proposition \ref{prop:rho-v}. Compute that for all $x \in M_N - \bigcup_{k=1}^n (\proj \circ \alpha)^{-1}(V_k)$,
\[
	(\proj\circ\alpha)\circ \rho_w(f)(x) = \proj \circ \rho_v([\alpha] \circ f \circ [\alpha^{-1}]) \circ \alpha(x) = (\proj \circ \alpha)(x)
\]
because $\proj\circ \rho_v([\alpha] \circ f \circ [\alpha^{-1}]) = \proj$ on $M_N - \bigcup_{k=1}^n \proj^{-1}(V_k)$ by Proposition \ref{prop:rho-v}\ref{prop:rho-v-2}. Hence $\rho_w(\Lambda_w)$ almost preserves $\proj\circ \alpha$, which is holomorphic for some complex structure on $M_N$. Finally, compute by Proposition \ref{prop:rho-v}\ref{prop:rho-v-1} that for any $f \in \Lambda_w$,
\[
	q(\alpha^{-1} \circ \rho_v([\alpha] \circ f \circ [\alpha^{-1}]) \circ \alpha) = [\alpha^{-1}] \circ (q \circ \rho_v)([\alpha] \circ f \circ [\alpha^{-1}]) \circ [\alpha] = f.\qedhere
\]
\end{proof}

If $2 \leq N \leq 8$, Theorem \ref{thm:parabolics} holds for any primitive, isotropic class in $H_2(M_N; \Z)$. 
\begin{cor}\label{cor:parabolics}
Let $2 \leq N \leq 8$ and let $w \in H_2(M_N; \Z)$ be any primitive, isotropic class. There exists a homomorphism $\rho_w: \Lambda_w \to \Diff^+(M_{n+1})$ such that the following diagram commutes: 
\begin{figure}[H]
\centering
\begin{tikzcd}
& \Diff^+(M_{N}) \arrow[d, "q"] \\
\Lambda_w \arrow[ru, "\rho_w", dashed] \arrow[r, hook] & \Mod(M_{N})                  
\end{tikzcd}
\end{figure}
Moreover, the image $\rho_w(\Lambda_w)$ almost preserves a holomorphic genus-$0$ Lefschetz fibration $p: M_N \to \CP^1$ whose generic fiber represents the homology class $w$.
\end{cor}
\begin{proof}
If $2 \leq N \leq 8$ and $w \in H_2(M_N; \Z)$ is a primitive, isotropic class then Lemma \ref{lem:isotropic}\ref{lem:minimal-genus} says that the minimal genus of $w$ is $0$. Now apply Theorem \ref{thm:parabolics}.
\end{proof}

\section{Theorem \ref{thm:individual-stab-v}: individual parabolic elements in $\Mod(M_{N})$}\label{sec:individual-stab-v}

In this section we prove Theorem \ref{thm:individual-stab-v} using the diffeomorphisms constructed in Section \ref{sec:thm-parabolics}. The following lemma considers the subgroup $S_n \leq \OO(n)(\Z) \leq \Stab(v)$ given by permuting the classes $e_1, \dots, e_n$. 
\begin{lem}\label{lem:transpositions}
For each $1 \leq k \leq n-1$, there exist $s_k \in \Diff^+(M_{N})$ and $\tau_k \in \Diff^+(\CP^1)$ such that 
\begin{enumerate}[(a)]
\item $\proj \circ s_k = \tau_k \circ \proj$, and \label{lem:trans-1}
\item $[s_k] = \RRef_{e_k-e_{k+1}} \in \Mod(M_{N}). $\label{lem:trans-2}
\end{enumerate}
\end{lem}
\begin{proof}
Let
\[
	A := \begin{pmatrix}
	-1 & 0 \\ 0 & 1
	\end{pmatrix} \in \PGL_2(\C)
\]
so that $A$ has order $2$ and $A([1:1]) = [-1:1] \in \CP^1$. There exists a neighborhood $D$ diffeomorphic to a disk $\mathbb D^2$ of the path $\{[t:1] : t \in [-1, 1]\} \subseteq \CP^1$ that is preserved by $A$. Let $\iota_k: D \hookrightarrow \CP^1$ be a smooth embedding with image contained in $B_k - U_k$ and
\[
	\iota_k([-1:1]) = [a_k:1], \quad \iota_k([1:1]) = [a_{k+1}:1], 
\]
so that $\iota_k$ is holomorphic if restricted to small neighborhoods of $[1:1]$ and $[-1:1]$ in $D$. Now let $\tau_k \in \Diff^+(\CP^1)$ be a diffeomorphism such that
\[
	\tau_k = \begin{cases}
	\iota_k \circ A \circ \iota_k^{-1} & \text{ on }\iota_k(D)\subseteq B_k - U_k, \\
	\Id & \text{ on }\CP^1 - B_k.
	\end{cases}
\]
Consider the diffeomorphism $s:(X, Y) \mapsto (\tau_k(X), Y)$ of $\CP^1 \times \CP^1$ which extends to a diffeomorphism $s_k$ of $M_N$ because $s$ is holomorphic on a neighborhood of $\proj^{-1}([a_i:1])$ for all $1 \leq i\leq n$. By construction, $\proj \circ s_k = \tau_k \circ \proj$. Moreover, if $i \neq k$ or $k+1$ then $s_k$ acts as the identity on $e_i$ but $s_k(e_k) = e_{k+1}$ and $s_k(e_{k+1}) = e_k$. Hence $[s_k] = \RRef_{e_k - e_{k+1}}$ by Lemma \ref{lem:stab-v-determined} because $[s_k]\in \Stab(v)$. 
\end{proof}

We may assume that for all $1 \leq k \leq n-1$, the choice of $V_{k+1} \subseteq \CP^1$ satisfies $\tau_k(V_k) = V_{k+1}$. This also implies that $\tau_k(V_{k+1}) = V_k$ because $\tau_k|_{B_k-U_k}$ has order $2$. 
\begin{proof}[Proof of Theorem \ref{thm:individual-stab-v}]
The theorem holds for $N = 1$ because then $\Stab(w) = 1$. Now assume that $N \geq 2$ and that $w = v$. Since $h \in \Stab(v)$, we may write $h = f \circ \sigma$ where $f \in \Lambda_v$ and $\sigma \in \Aut(\Z\{e_1, \dots, e_n\}, Q_{M_N}) \cong \OO(n)(\Z)$ by Lemma \ref{lem:ses}. Furthermore, $\sigma$ can be written as a product $[r] \circ [s] \in \Aut(\Z\{e_1, \dots, e_n\}, Q_{M_N})$ where
\[
	r \in \langle r_k : 1 \leq k \leq n \rangle \leq \Diff^+(M_N), \qquad s \in \langle s_k : 1 \leq k \leq n-1 \rangle \leq \Diff^+(M_N),
\]
by Remark \ref{rmk:rk} and Lemma \ref{lem:transpositions}\ref{lem:trans-2}. Let
\[
	\varphi := \rho_v(f) \circ r \circ s \in \Diff^+(M_N)
\]
where $\rho_v: \Lambda_v \to \Diff^+(M_N)$ is the homomorphism from Proposition \ref{prop:rho-v}. By construction, $[\varphi] = h$.

Note that 
\[
	\proj\circ\rho_v(f)\circ r|_{M_N-\bigcup_{i=1}^n \proj^{-1}(V_i)} = \proj|_{M_N-\bigcup_{i=1}^n \proj^{-1}(V_i)}
\]
by Proposition \ref{prop:rho-v}\ref{prop:rho-v-2} and by Remark \ref{rmk:rk}. By Lemma \ref{lem:transpositions}\ref{lem:trans-1}, there exists $\tau \in \Diff^+(\CP^1)$ such that $\proj\circ s = \tau \circ \proj$ and $\tau$ preserves $\bigcup_{i=1}^n \proj^{-1}(V_i)$. Hence 
\[
	\proj\circ \varphi|_{M_N-\bigcup_{i=1}^n \proj^{-1}(V_i)}= \proj\circ (\rho_v(f) \circ r \circ s)|_{M_N-\bigcup_{i=1}^n \proj^{-1}(V_i)} = \tau \circ \proj |_{M_N-\bigcup_{i=1}^n \proj^{-1}(V_i)},
\]
which shows that $\varphi$ almost preserves $\proj$. 

Take any other primitive, isotropic class $w \in H_2(M_N; \Z)$ with minimal genus $0$; we proceed similarly as in the proof of Theorem \ref{thm:parabolics}. Apply Li--Li \cite[Theorem 4.2]{li--li} to obtain $\alpha \in \Diff^+(M_N)$ such that $[\alpha](w)= v$ and $[\alpha] \circ g \circ [\alpha^{-1}] \in \Stab(v)$. There exists a diffeomorphism $\varphi \in \Diff^+(M_N)$ almost preserving $\proj: M_N \to \CP^1$ with $[\varphi] = [\alpha] \circ g \circ [\alpha^{-1}]$. Then $\alpha^{-1} \circ \varphi \circ \alpha$ almost preserves $\proj \circ \alpha$, and $[\alpha^{-1} \circ \varphi \circ \alpha] = h$
\end{proof}

\begin{cor}\label{cor:individual-stab-v}
Let $N \leq 8$ and let $w \in H_2(M_N)$ be any primitive, isotropic class. For any $h \in \Stab(w)$, there exists $\varphi \in \Diff^+(M_N)$ almost preserving some holomorphic genus-$0$ Lefschetz fibration $p: M_N \to \CP^1$ whose generic fiber represents the homology class $w$ such that $[\varphi] = h \in \Mod(M_N)$. 
\end{cor}
\begin{proof}
If $N \leq 8$ and $w \in H_2(M_N; \Z)$ is a primitive, isotropic class then Lemma \ref{lem:isotropic}\ref{lem:minimal-genus} says that the minimal genus of $w$ is $0$. Now apply Theorem \ref{thm:individual-stab-v}.
\end{proof}

\small
\bibliographystyle{alpha}
\bibliography{parabolics}

\bigskip
\noindent
Seraphina Eun Bi Lee \\
Department of Mathematics \\
University of Chicago \\
\href{mailto:seraphinalee@uchicago.edu}{seraphinalee@uchicago.edu}

\end{document}